\documentclass[twocolumn]{autart}
\usepackage{amsmath}
\usepackage{amsfonts}
\usepackage{graphicx}
\graphicspath{{./images/}}
\newtheorem{lemma}{Lemma}
\newtheorem{theorem}{Theorem}
\newtheorem{corollary}{Corollary}
\newtheorem{definition}{Definition}
\newtheorem{remark}{Remark}

\newtheorem{example}{Example}
\newtheorem{proposition}{Proposition}

\usepackage{caption}
\usepackage{subcaption}
\usepackage{mathtools}
\usepackage[final]{changes}
\usepackage{algpseudocode}
\usepackage{algorithm}
\newcommand{\defeq}{\overset{\mathrm{def}}{=}}

\newenvironment{proof}{\paragraph*{Proof:}}{\hfill$\square$}
\begin{document}
\begin{frontmatter}
\title{Logarithmically Completely Monotonic Rational Functions}
\author[1]{Hamed Taghavian}\ead{hamedta@kth.se},    
\author[2]{Ross Drummond}\ead{ross.drummond@sheffield.ac.uk.},               
\author[1]{Mikael Johansson}\ead{mikaelj@kth.se}  
\address[1]{KTH Royal Institute of Technology, Stockholm, Sweden.}  
\address[2]{University of Sheffield, Sheffield, UK.}             

\begin{abstract}
\replaced{This paper studies the class of logarithmically completely monotonic (LCM) functions.}{In this paper, the class of logarithmically completely monotonic (LCM) functions are studied.} \replaced{These functions}{This class of functions} play an important role in characterising \deleted{the class of} {externally positive} linear systems which find applications in \deleted{several} important control problems \replaced{such as}{including} non-overshooting reference tracking. Conditions are proposed to ensure a rational function is LCM, a result that enables the known space of linear continuous-time {externally positive} systems to be enlarged and an efficient and optimal pole-placement procedure for the monotonic tracking controller synthesis problem to be developed. The presented conditions are shown to be less conservative than existing approaches whilst being computationally tractable.
\end{abstract}
\end{frontmatter}
\section{Introduction}\label{sec:intro}
The class of completely monotonic functions \cite{Wid2015} play an important role in \deleted{many} diverse application areas, \replaced{ranging from}{including} probability theory \cite{Kim1974}\replaced{ and}{,} physics \cite{StW2020} \replaced{to}{and} combinatorics \cite{Bal1994,che2005}. \replaced{In automatic control,}{Moreover, for control engineers,} {complete monotonicity of transfer functions characterises the class of externally positive systems}\replaced{, that is, }{; being those}  systems {whose impulse responses are non-negative for all times,} mapping non-negative inputs to non-negative outputs~\cite{BSFG2017}. {Systems that preserve positivity in this way 
{are often encountered} in physical systems operating on ``positive quantities'' like ion concentrations and population sizes, as encountered in, for example, chemical systems \cite{LAS2007}, biological systems \cite{BSFG2017} and transportation systems \cite{lunzpa}. {
Many practical control systems also have to be designed  to ensure the positivity of certain signals. For example, many important transient response control problems require the appropriate transfer functions to be externally positive~\cite{MDB2009}. Therefore, studying externally positive systems has the potential to improve the capabilities of many control problems. For each of these problems, their solutions are currently limited by our inability to fully characterise the class of externally positive systems. As remarked earlier, externally positive systems can be characterized as completely monotonic transfer functions~\cite{Bry1999}. Therefore, the class of logarithmically completely monotonic (LCM) rational functions introduced in~\cite{Bal1994} become relevant, as all these functions are also completely monotonic.}
In this paper, we build upon this work and present several novel results for LCM functions. In particular, we provide a set of conditions that characterise a larger family of LCM rational functions than what was previously available, and present generalisations of the results in~\cite{Bal1994} to systems with complex zeros and poles.}

{It appears that no simple\deleted{, yet tight,} condition\deleted{s} \deleted{(posed in either the time or frequency domains)} can be obtained for characterising \added{all} {externally positive systems in terms of their transfer function coefficients or pole-zero locations}}~\cite{ScL2020}. \replaced{In}{However, in} spite of such difficulties, several sufficient conditions \replaced{that}{to} ensure \added{that} a given transfer function is externally positive have been developed (see, \emph{e.g.},   \cite{DTD2019,LiB2008,lunzpa,cdc2}), all of which have some degree of conservatism. In this paper,  sufficient conditions for guaranteeing that a transfer function is externally positive are also developed, but the framework of LCM rational functions is exploited to reduce the conservatism of earlier results. This reduced conservatism is demonstrated through both relaxed assumptions on the system and better performance in numerical examples\deleted{ numerical examples and also by their relaxed assumptions on the system properties}. For example, in contrast to~\cite{DTD2019}, the conditions provided in this paper are applicable to systems with multiple dominant poles and expose all externally positive systems of up to order two. Furthermore, as opposed to~\cite{LiB2008,lunzpa,cdc2}, the presented conditions can also be applied to transfer functions with complex-conjugate zeros and poles.

{By using the developed LCM conditions to enlarge the class of systems that can be guaranteed to be externally positive, our results are expected to offer improved solutions to several problems in control theory, including transient response control~\cite{MDB2012}, Zames-Falb multiplier search methods~\cite{TuD2019} and the problem of synthesizing controllers that ensure non-overshooting reference tracking~\cite{TaJ2020}.} Like the characterization problem, designing output feedback control loops {with a non-negative closed-loop impulse response} is \replaced{an}{also an old yet} open problem in control theory \cite{ScL2020}, even though it plays an important role in application areas such as traffic control~\cite{lunzpa} and robotics~\cite{ECL1993}. Specifically, it was shown in \cite{ScL2021} that designing adaptive cruise control systems based on {this property} instead of string stability can guarantee collision avoidance in vehicle platoons. Furthermore, imposing {non-negativity on the impulse response samples of the closed loop system} is a popular way to eliminate both overshoots and undershoots in the system response~\cite{MTNS}. {The importance of non-overshooting reference tracking has led to the development of several interesting techniques to address this problem, such as  the state-feedback controllers of \cite{ScN2010}, the three-part controller structures of \cite{Dar2003} and the integrating controllers of \cite{ScL2020} to name just a few.} In this paper, we introduce a  pole-placement technique to synthesize feedback controllers that can minimize a general convex objective function whilst ensuring the closed-loop system {to be externally positive}, hence guaranteeing non-overshooting reference tracking. The proposed design technique has many advantages compared to the state-of-the art. Contrary to~\cite{ScN2010}, the presented conditions can be used to design output feedback controllers; in contrast to~\cite{Dar2003}, it is a computationally tractable and constructive approach based on convex optimization; and different from~\cite{ScL2020}, the plant under control is not required to have a non-negative impulse response. 

The paper is organized as follows. We clarify the relationship between LCM functions, completely monotonic functions and externally positive transfer functions in Section~\ref{sec:pre}. In Section~\ref{sec:lcm}, we derive a number of conditions that characterize LCM transfer functions. These conditions are either necessary or sufficient or both. In Section~\ref{sec:apps} we focus on the implication of the obtained results in control theory. In particular, we use the LCM conditions to characterize {the family of externally positive} systems and design optimal output feedback controllers ensuring a monotonic tracking. Finally conclusive remarks are presented in Section~\ref{sec:con}.

{
\subsection{Notation}
We use the following notation. The set of negative numbers is $\mathbb{R}_{<0}$ and the set of $n$-dimensional vectors with negative components is $\mathbb{R}^n_{<0}$. The Laplace transform is denoted by $\mathfrak{L}\lbrace \cdot \rbrace$ and its inverse by $\mathfrak{L}^{-1}\lbrace \cdot\rbrace$. The $k$-th derivative of function $H(s)$ is denoted by $H^{(k)}(s)$, while the $m$-th power of $H(s)$ is denoted by $H^{m}(s)$. For $z\in\mathbb{C}$, $\operatorname{Re}(z)$ denotes the real part, $\operatorname{Im}(z)$ denotes the imaginary part and $\angle (z)$ denotes its angle. Assuming $x,y\in\mathbb{R}^n$, weak majorization of $y$ by $x$ is denoted by $x \succ_w y$ and vector $x$ with the same components sorted in descending order is denoted by $x^{\downarrow}$. The $n$-fold Cartesian product of a set $I$ with itself is denoted by $I^n=I\times I\cdots \times I$. 
}

\section{Preliminaries}\label{sec:pre}
We focus on rational functions of the form
\begin{align}\label{eqn:transfer_function}
    H(s)&=\frac{B(s)}{A(s)}=K\frac{ \prod_{i=1}^{m} (s-z_i)}{\prod_{i=1}^n (s-p_i)}\nonumber\\
    &=\frac{ b_0 s^{n}+b_1 s^{n-1}+ \cdots +b_{n}}{ s^{n}+a_1 s^{n-1}+ \cdots +a_{n}}
\end{align}
where $K\neq 0$. For convenience, we introduce the vectors $z\in\mathbb{C}^m$ and $p\in \mathbb{C}^n$ whose components are the zeros $z_i$ and the poles $p_i$, respectively. Note that in this notation, $b_0=b_1=\cdots =b_{n-m-1}=0$ and $b_{n-m}=K$. 

\subsection{(Logarithmically) completely monotonic functions}

We begin by introducing the classes of completely monotonic and logarithmically completely monotonic functions and clarify the relations between the two.
%
\begin{definition}[\cite{Wid2015}]\label{def:CM}
A function $H:\mathbb{C} \to\mathbb{C}$ is called completely monotonic (CM) on $I\subseteq \mathbb{R}$ if
\begin{equation}\label{eqn:CM}
(-1)^{k}H^{(k)}(s)\geq 0
\end{equation}
holds for all $k\in \mathbb{N}_0$ and $s\in I$.
\end{definition}

\begin{definition}[\cite{QiC2004}]\label{def:LCM}
A function $H:\mathbb{C} \to\mathbb{C}$ is said to be logarithmically completely monotonic (LCM) on $I\subseteq \mathbb{R}$, if $H(s)>0$ and
$$
(-1)^{k}[\log H(s)]^{(k)}\geq 0
$$
holds for all $k\in \mathbb{N}$ and $s\in I$.
\end{definition}
It can be shown that LCM functions form a proper subset of CM functions \cite{QiC2004}. {In particular, a function  $H:\mathbb{C} \to\mathbb{C}$ is LCM if and only if $H^{1/k}$ is CM\deleted{, not only for $k=1$ but} for all $k\in \mathbb{N}$ \cite{BeC2004}.} 

\subsection{Externally positive systems}
Logarithmic complete monotonicity is used in this paper to characterise systems with non-negative impulse responses. 
\begin{definition}[\cite{BSFG2017}]\label{def:EP}
{The transfer function $H:\mathbb{C} \to\mathbb{C}$ is called externally positive\deleted{\cite{BSFG2017}} if its inverse Laplace transform satisfies $h(t)=\mathfrak{L}^{-1}\lbrace H(s) \rbrace \geq 0$ for all $t \in [0,+\infty)$. \deleted{As in, if it is a transfer function with a non-negative impulse response.}}
\end{definition}
A linear time-invariant continuous-time system with an externally positive transfer function has a non-negative impulse response. Externally positive functions are closely related to CM functions through the Bernstein theorem \cite{SSV2012}. In particular, if we choose $I=(r_0,+\infty)$ where $r_0\geq \sigma(H)$ and $\sigma(H)=\max_i\left\lbrace \operatorname{Re}(p_i) \right\rbrace$ is the pole abscissa, then the transfer function $H(s)$ given by (\ref{eqn:transfer_function}) is completely monotonic on $I$  if and only if it is externally positive~\cite{EJC,Bry1999}. This means that every CM function is the Laplace transform of a non-negative function in the time domain and vice versa. 

Most importantly, if $H(s)$ is LCM, then it is CM, and hence externally positive. Therefore in this paper, we are interested in algebraic conditions on the zeros and poles of rational functions $H:\mathbb{C}\to \mathbb{C}$ of the form (\ref{eqn:transfer_function}) that ensure that $H(s)$ {is} LCM. These conditions also ensure the transfer function $H(s)$ is externally positive and, therefore, help characterise a large family of externally positive systems.

\subsection{Majorization theory}
{Our novel characterizations of LCM functions are derived based on majorization theory.}
Majorization theory, determines how ``spread out'' the components of two vectors are with respect to each other and is an important tool in \replaced{the study of}{formalization of} mathematical inequalities~\cite{MOA2010}. Notably\added{,} as shown in \cite{DTD2019}, majorization can be used to \replaced{derive an elegant characterization of a subset of}{neatly characterize a major set of} externally positive rational functions \deleted{(\ref{eqn:transfer_function})} in terms of their zeros and poles. 

\begin{definition}
For a vector $z\in {\mathbb R}^n$, let $z^{\downarrow}$ be the vector of the same components sorted in descending order. We say that $x\in {\mathbb R}^n$ weakly majorizes $y \in {\mathbb R}^n$, denoted $x\succ_w y$, if
$$
\sum_{i=1}^k x^{\downarrow}_i \geq \sum_{i=1}^k y^{\downarrow}_i
$$
{for all $k=1,2\cdots,n$.}
\end{definition} 
{The following proposition presents operators that preserve a majorization relation.}

\begin{proposition}[\cite{MOA2010}, Theorem A.2]\label{prop:preserving_majorization}
{Let $I\subseteq \mathbb{R}$, $x,y\in I^n$ and $g:I\to \mathbb{R}$ be convex and increasing on $I$.} If $x\succ_w y$, then
$$
g(x) \succ_w g(y)
$$
where $g(x)=\begin{bmatrix}g(x_1),\cdots,g(x_n)\end{bmatrix}^T$.
\end{proposition}

{A special case of Proposition~\ref{prop:preserving_majorization} is the celebrated Karamata's inequality:}

\begin{proposition}[\cite{wey1949}]\label{prop:Karamata}
{Let $I\subseteq \mathbb{R}$, $x,y\in I^n$ and $g:I\to \mathbb{R}$ be convex and increasing on $I$.} If $x\succ_w y$, then
$$
\sum_{i=1}^n g(x_i) \geq \sum_{i=1}^n g(y_i)
$$
\end{proposition}

Propositions~\ref{prop:Karamata} and \ref{prop:preserving_majorization} will be used in the next section to derive conditions under which (\ref{eqn:transfer_function}) is LCM.

\section{LCM conditions}\label{sec:lcm}
In this section, we are interested in conditions under {which rational functions $H:\mathbb{C}\to \mathbb{C}$ of the form of  (\ref{eqn:transfer_function}) are LCM}. Since (\ref{eqn:transfer_function}) is uniquely determined by $K\in \mathbb{R}$, $z\in\mathbb{C}^m$ and $p\in\mathbb{C}^n$, we seek such conditions in terms of the function{'s} zeros, poles and static gain. {We begin by introducing conditions that are both necessary and sufficient for (\ref{eqn:transfer_function}) to be LCM.}

\subsection{A full characterization}
The following lemma offers a characterization for when \deleted{the main characterizing condition for} (\ref{eqn:transfer_function}) \replaced{is}{to be} LCM in the general case\deleted{,} where $z\in \mathbb{C}^m$ and $p\in \mathbb{C}^n$.

\begin{lemma}\label{lem:N&S_sumexp}
{ $H:\mathbb{C}\to \mathbb{C}$ defined in 
(\ref{eqn:transfer_function}) is} logarithmically completely monotonic if and only if $K>0$ and
\begin{equation}\label{eqn:exp(pt)>exp(zt)}
    \sum_{i=1}^n \exp(p_i t) \geq \sum_{i=1}^m \exp(z_i t)
\end{equation}
holds for all $t\in [0,+\infty)$.
\end{lemma}
\begin{proof}
    First note that $H(s)>0$ holds for all $s>\sigma(H)$ if and only if $K>0$. Calculating the successive derivatives of $\log H(s)$ reveals that (\ref{eqn:transfer_function}) is LCM if and only if
    \begin{align}\label{eqn:logF^(k)}
        &(-1)^k[\log H(s)]^{(k)}\nonumber\\
        &=(k-1)!\left( \sum_{i=1}^n \frac{1}{(s-p_i)^k} -\sum_{i=1}^m \frac{1}{(s-z_i)^k}\right) \nonumber\\
        &=(-1)^{k-1}G(s)^{(k-1)} \geq 0
    \end{align}
    holds for all $k \in \mathbb{N}$, where
    \begin{equation}\label{eqn:mixed_relax}
    G(s)=\sum_{i=1}^n \frac{1}{s-p_i} -\sum_{i=1}^m \frac{1}{s-z_i}
    \end{equation}
    According to (\ref{eqn:CM}), this is equivalent to $G(s)$ being completely monotonic or equivalently, externally positive \cite{EJC,Bry1999}. This is realized if and only if $g(t)=\mathfrak{L}^{-1}\lbrace G(s)\rbrace$ is non-negative which yields inequality (\ref{eqn:exp(pt)>exp(zt)}).
\end{proof}

{Lemma~\ref{lem:N&S_sumexp} is a significant improvement of a corresponding result in \cite{Bal1994}, which only proved the `if' part in the special case when $z\in\mathbb{R}_{<0}^{m}$ and $p\in\mathbb{R}_{<0}^{n}$.} The lemma immediately narrows down the family of LCM transfer functions (\ref{eqn:transfer_function}) by revealing a few trivial cases in which the function cannot be LCM.

\begin{proposition}\label{prop:necessary_conditions}
The following conditions are necessary for $H:\mathbb{C}\to \mathbb{C}$ of the form (\ref{eqn:transfer_function}) to be LCM:
\begin{enumerate}
    \item[(a)] $n\geq m$. 
    \item[(b)]  $\max{\operatorname{Re}(p_i)} \geq \max{\operatorname{Re}(z_i)}$.
    \item[(c)] $\sum_{i=1}^n p_i \geq \sum_{i=1}^m z_i$ when $n=m$.
\end{enumerate}
\end{proposition}
\begin{proof}
Condition (a) is follows from setting $t=0$ in (\ref{eqn:exp(pt)>exp(zt)}).

Condition (b) is proved by contradiction. Assume that $\max{\operatorname{Re}(z_i)}>\max{\operatorname{Re}(p_i)}$. {Dividing} both sides of (\ref{eqn:exp(pt)>exp(zt)}) by $\exp(\max{\operatorname{Re}(z_i)}t)$ would violate the inequality {as} $t\to +\infty$. {Hence,} if (\ref{eqn:transfer_function}) is LCM, then $\max{\operatorname{Re}(z_i)}\leq \max{\operatorname{Re}(p_i)}$ holds true.

\noindent Condition (c) has also been mentioned in \cite{Bal1994}. Here, we use an alternative way to show its necessity: When $n=m$, both sides of (\ref{eqn:exp(pt)>exp(zt)}) are equal at $t=0$. Therefore, for (\ref{eqn:exp(pt)>exp(zt)}) to hold for all $t\geq 0$, it has to hold for the derivatives at $t=0$ as well, namely for
$$
\sum_{i=1}^n p_i\exp(p_i t) \geq \sum_{i=1}^n z_i\exp(z_i t)
$$
at $t=0$. This proves the third inequality.
\end{proof}

The main drawback of Lemma~\ref{lem:N&S_sumexp} is that it requires validating inequality (\ref{eqn:exp(pt)>exp(zt)}) {on the infinite range $t \in [0, \infty)$}. {In practice, checking \eqref{eqn:exp(pt)>exp(zt)} is often implemented (naively) by sampling $t \in [0, \infty)$ and evaluating the condition at the chosen time instants. But this procedure does not provide any formal guarantees of the inequality holding in the semi-infinite range.} 
{A potential solution to this problem involves restricting the poles $p$ and zeros $z$ to some subsets. For example, it can be shown that if $z$ and $p$} have only real commensurable components \emph{i.e.} assuming that there is some $\gamma\in (0,+\infty)$ such that $z_i/\gamma, p_i/\gamma \in \mathbb{N}_0$, one may convert (\ref{eqn:exp(pt)>exp(zt)}) to the following polynomial inequality
\begin{equation}\label{eqn:N&S_poly}
    P(x)=\sum_{i=1}^n x^{p_i/\gamma} - \sum_{i=1}^m x^{z_i/\gamma} \geq 0, \quad \forall x\in [1,+\infty)
\end{equation}
{after applying the change-of-variables} $x=\exp(\gamma t)$. For fixed poles and zeros, Condition (\ref{eqn:N&S_poly}) can be {readily} checked {via semi-definite programming} using the Markov-Luk{\'a}cs theorem \cite{BPS2012,MDB2009}. However, in addition to {restricting the zeros and poles to be real commensurable, this method does not allow for any convenient way to optimize over the pole and zero locations.} 

\subsection{Sufficient conditions}
The aim of this section is to obtain sufficient conditions for (\ref{eqn:transfer_function}) to be LCM. In contrast with Lemma~\ref{lem:N&S_sumexp}, these conditions are no longer necessary for logarithmic complete monotonicity, but they can be checked using a finite number of inequalities which makes them useful for control synthesis. 

\subsubsection{Real zeros and poles}
In~\cite{Bal1994}, it was shown that a rational transfer function (\ref{eqn:transfer_function}) with $K>0$ is LCM if 
\begin{equation}\label{eqn:Ball}
    p,z\in (-\infty,0]^n \textnormal{ and } p \succ_w z.
\end{equation}
This is an elegant condition, but {it can sometimes be conservative and it is restricted to systems with real zeros and poles.}
In the following theorem, we present a sufficient condition for LCM functions which is always less conservative than (\ref{eqn:Ball}) and has a straight-forward extension to complex zeros and poles. {Compared to \eqref{eqn:Ball}, the main point of differentiation with this theorem is the inclusion of the terms $\mu$ and $\delta$, with $\mu$ helping to reduce the conservatism  by scaling the conditions and $\delta$ simply introducing a translation of the poles and zeros to allow positive values to be considered.} 

\begin{theorem}\label{thm:p^mu>z^mu}
The function $H(s)$ in (\ref{eqn:transfer_function}) {with $z_i,\, p_i\in \mathbb{R}$}, $m=n$ and $K>0$  
is LCM if 
\begin{align}
    (p+\delta) ^{\mu} &\succ_w  (z+\delta)^{\mu} \nonumber
    \intertext{and}
    \sum_{i=1}^n (p_i+\delta)^k  &\geq \sum_{i=1}^n (z_i+\delta)^k \quad \forall k\in \left\{1, \dots, \mu-1 \right\}\label{eqn:psum(p)>psum(z),1_mu}
\end{align}
holds for some $\mu \in \mathbb{N}$ and some $\delta>-\min_{i,j}\lbrace p_i,z_j\rbrace$.
\end{theorem}
\begin{proof}
{Since $n=m$, inequality (\ref{eqn:psum(p)>psum(z),1_mu}) is satisfied with equality for $k=0$. By assumption, the inequality also holds for $1\leq k\leq \mu -1$. To show that (\ref{eqn:psum(p)>psum(z),1_mu}) is met for $k\geq \mu$, recall that Proposition~\ref{prop:Karamata} states that the relation $(p+\delta) ^{\mu} \succ_w (z+\delta) ^{\mu}$ implies that
\begin{equation}\label{eqn:sumgp>sumgz}
    \sum_{i=1}^n g\left((p_i+\delta)^\mu\right) \geq \sum_{i=1}^n g\left((z_i+\delta)^\mu\right)
\end{equation}
for every convex and increasing function {$g:I\to \mathbb{R}$}. Choosing {$I=[0,+\infty)$ and} $g(x)=x^{k/\mu}$ in (\ref{eqn:sumgp>sumgz}) proves inequality (\ref{eqn:psum(p)>psum(z),1_mu}) for all $k\in\lbrace  \mu,  \mu+1,\cdots\rbrace$, because for these values of $k$, $g$ is convex and increasing {on $I$} and $(z_i+\delta),(p_i+\delta)\in I$ holds for all $i$ as $\delta>-\min_{i,j}\lbrace p_i,z_j\rbrace$. Hence, since (\ref{eqn:psum(p)>psum(z),1_mu}) holds for all $k\geq 0$,} one may write
\begin{equation}\label{eqn:sum(pt)^k/k!>sum(zt)^k/k!}
\sum_{i=1}^n \frac{(p_i+\delta)^k t^k}{k!} \geq \sum_{i=1}^n \frac{(z_i+\delta)^k t^k}{k!}
\end{equation}
for all $k\in\mathbb{N}_0$ and $t\in[0,\infty)$. {Summing} both sides of (\ref{eqn:sum(pt)^k/k!>sum(zt)^k/k!}) over $k\in\mathbb{N}_0$ results in (\ref{eqn:exp(pt)>exp(zt)}){, and, therefore the function $H:\mathbb{C}\to \mathbb{C}$ of the form \eqref{eqn:transfer_function}} is LCM due to Lemma~\ref{lem:N&S_sumexp}.
\end{proof}

Considering $\mu=1$ in Theorem~\ref{thm:p^mu>z^mu} recovers Ball's condition (\ref{eqn:Ball}). {The following proposition indicates that the conditions in Theorem~\ref{thm:p^mu>z^mu} are, in general, less conservative than the condition (\ref{eqn:Ball}) from \cite{Bal1994}.}

{ 
\begin{proposition}\label{prop:mu}
Increasing $\mu$ in Theorem~\ref{thm:p^mu>z^mu} enlarges the set of rational functions (\ref{eqn:transfer_function}) that can be proven to be LCM.
\end{proposition}}
\begin{proof}
Let $\mu_2>\mu_1$ and $p$ and $z$ satisfy the hypotheses of Theorem~\ref{thm:p^mu>z^mu} with $\mu=\mu_1$. Then, $(p+\delta)^{\mu_1} \succ_w (z+\delta)^{\mu_1}$ and invoking Proposition~\ref{prop:preserving_majorization} with {$I=[0,+\infty)$ and} $g(x)=x^{\mu_2/\mu_1}$ establishes that $(p+\delta)^{\mu_2}\succ_w (z+\delta)^{\mu_2}$. In addition, using Proposition~\ref{prop:Karamata} with $g(x)=x^{k/\mu_1}$ reveals that (\ref{eqn:psum(p)>psum(z),1_mu}) holds for $k\in\lbrace \mu_1 ,\cdots, \mu_2 -1 \rbrace$. We can thus conclude that $p$ and $z$ also satisfy the hypothesis of Theorem~\ref{thm:p^mu>z^mu} with $\mu=\mu_2$. {This indicates that the set of rational functions (\ref{eqn:transfer_function}) satisfying Theorem~\ref{thm:p^mu>z^mu} with $\mu=\mu_1$ is a subset of that with $\mu=\mu_2$.}
\end{proof}

We will now evaluate the sufficient conditions obtained in Theorem~\ref{thm:p^mu>z^mu} in a numerical example. {This example demonstrates} that the conservatism can disappear even with a finite value of $\mu$.

\begin{example}\label{ex:5th_order}
Consider the following transfer function
$$
{H(s)}=\frac{(s+2)(s+3)(s+5)(s+6)(s+8)}{(s-p_1)(s-p_2)(s+1)(s+4)(s+7)}
$$
The values of $p_1$ and $p_2$ for which the conditions in Theorem~\ref{thm:p^mu>z^mu} guarantee the system to be LCM are shown in Figure~\ref{fig:5th_order}, as the transparent colored regions on the right side of the specified boundaries. As anticipated by Proposition~\ref{prop:mu}, increasing $\mu$ enlarges the set of LCM functions detected by Theorem~\ref{thm:p^mu>z^mu}. Interestingly, the conservatism disappears with $\mu=3$ and the set of LCM functions characterized by the sufficient conditions of Theorem~\ref{thm:p^mu>z^mu} coincides with the exact set of LCM functions (and externally positive transfer functions). {This is evident in Figure~\ref{fig:5th_order} as the region specified with $\mu=3$ (red) includes all the pairs $(p_1,p_2)$ that satisfy
\begin{equation}\label{eqn:necexpos_sumpsumz}
\sum_{i=1}^n p_i \geq \sum_{i=1}^n z_i
\end{equation}
which is known to be necessary for $H(s)$ to be externally positive \cite{Bal1994}.
}
\end{example}

\begin{figure}
\begin{center}
        \includegraphics[width=0.8\linewidth]{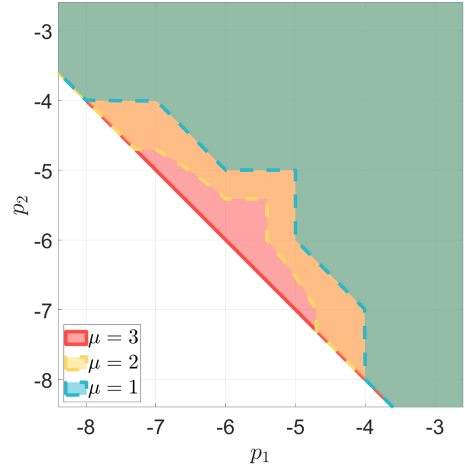}
	\caption{LCM regions found in Example~\ref{ex:5th_order}, by using different values of $\mu$ in Theorem~\ref{thm:p^mu>z^mu}.}
	\label{fig:5th_order}
\end{center}
\end{figure}

\subsubsection{Complex zeros and poles}
Extension of the condition (\ref{eqn:Ball}) from~\cite{Bal1994} to strictly proper systems ($m<n$) is straightforward by assuming that the missing zeros to be located at $z=-\infty$. Yet, (\ref{eqn:Ball}) is only applicable to systems with real zeros and poles.
Below, we extend Theorem~\ref{thm:p^mu>z^mu}, which contains (\ref{eqn:Ball}) as a special case, to strictly proper transfer functions as well as to complex zeros and poles. To this aim, we introduce $\theta\in\mathbb{R}^{n}$, $\phi\in \mathbb{R}^{m}$, $w,v\in\mathbb{R}^{n+m}$ based on $z$ and $p$:
\begin{align}\label{eqn:w,v}
    &\theta_i=\angle (p_i+\delta), \quad 1\leq i \leq n
    \nonumber\\
    &\phi_i=\angle (z_i+\delta), \quad 1\leq i \leq m
    \nonumber\\
    & w_i=\left\lbrace \begin{array}{ll}
        (p_i+\delta)^{\mu}, &  1\leq i \leq n_r\\
        0, & n_r+1 \leq i \leq n\\
        0, & n+1\leq i \leq n+m
    \end{array}\right.
    \nonumber\\
    & v_i=\left\lbrace \begin{array}{ll}
        0, & 1\leq i \leq n_r\\
        \vert p_i+\delta\vert^{\mu}, &  n_r+1\leq i \leq n\\
        \vert z_i+\delta \vert^{\mu}, & n+1\leq i \leq n+m
    \end{array}\right.
\end{align}
Here, $p_1,\dots,p_{n_r}$ are assumed to be real and $p_{n_r+1},\dots,p_{n}$ are the complex-conjugate poles. Note that $w$ contains all the {(shifted and scaled)} poles except the complex-conjugate ones and $v$ contains the absolute values of all the zeros and complex-conjugate poles. The next corollary constitutes our most general sufficient condition for a transfer function to be LCM.

\begin{corollary}\label{cor:complex}
Let $z\in\mathbb{C}^m$, $p\in\mathbb{C}^n$, $m\leq n$ and $K>0$. The function $H(s)$ is LCM if $w \succ_w v$ and
\begin{equation}\label{eqn:psum(p)>psum(z),1_mu_proper}
    \sum_{i=1}^n w_i^{k/\mu}+\sum_{i=1}^n v_i^{k/\mu}\cos(\theta_i k) \geq \sum_{i=n+1}^{n+m} v_i^{k/\mu}\cos(\phi_{i-n} k) 
\end{equation}
holds for all $k\in \lbrace 1,2,\cdots,\mu -1\rbrace$ and some $\mu \in \mathbb{N}$ and some $\delta>-\min_{i,j}\lbrace \operatorname{Re}(p_i),\operatorname{Re}(z_j)\rbrace$, where $\theta,\phi,w,v$ are defined in (\ref{eqn:w,v}).
\end{corollary}
\begin{proof}
We start by expanding the condition (\ref{eqn:exp(pt)>exp(zt)}) as
\begin{equation}\label{eqn:expanded_exp(pt)>exp(zt)}
    \sum_{k=0}^{+\infty}\sum_{i=1}^n (p_i+\delta)^kt^k/k!\geq \sum_{k=0}^{+\infty}\sum_{i=1}^m (z_i+\delta)^kt^k/k!
\end{equation}
where $t\in[0,\infty)$. According to Lemma~\ref{lem:N&S_sumexp}, logarithmic complete monotonicity can be assured by requiring the inequality (\ref{eqn:expanded_exp(pt)>exp(zt)}) to hold term-by-term for all corresponding powers of $k\in\mathbb{N}_0$, \emph{i.e.}
\begin{equation}\label{eqn:expanded_exp(pt)>exp(zt)_2}
    \sum_{i=1}^n (p_i+\delta)^k\geq \sum_{i=1}^m (z_i+\delta)^k
\end{equation}
This is equivalent to (\ref{eqn:psum(p)>psum(z),1_mu_proper}) noting the variable change (\ref{eqn:w,v}) and therefore, it holds for $k\leq \mu-1$ by assumption and it holds for $k=0$ because $n\geq m$. In order to show that (\ref{eqn:expanded_exp(pt)>exp(zt)_2}) also holds for $k\geq \mu$, we first write
\begin{equation}\label{eqn:corproof_step1}
\sum_{i=1}^n (p_i+\delta)^k \geq \sum_{i=1}^{n} w_i^{k/\mu}-\sum_{i=1}^{n} v_i^{k/\mu}
\end{equation}
Since $w \succ_w v$, it is deduced by Proposition~\ref{prop:Karamata} that
\begin{equation}\label{eqn:corproof_step2}
\sum_{i=1}^{n} w_i^{k/\mu}-\sum_{i=1}^{n} v_i^{k/\mu}  \geq
\sum_{i=n+1}^{n+m} v_i^{k/\mu}
\end{equation}
as the function $g(x)=x^{k/\mu}$ is convex and increasing on $x\in[0,+\infty)$ when $k\geq \mu$. Together, inequalities (\ref{eqn:corproof_step1}) and (\ref{eqn:corproof_step2}) imply that 
\begin{align*}
\sum_{i=1}^n (p_i+\delta)^k &\geq \sum_{i=n+1}^{n+m} v_i^{k/\mu} \\
                            &=\sum_{i=1}^{m} \vert z_i+\delta\vert^{k} \geq \sum_{i=1}^{m} (z_i+\delta)^{k}
\end{align*}
Thereby we have shown that condition (\ref{eqn:expanded_exp(pt)>exp(zt)_2}) holds for all $k\in\mathbb{N}_0$ and hence, that $H(s)$ is LCM.
\end{proof}

Note that Corollary~\ref{cor:complex} recovers Theorem~\ref{thm:p^mu>z^mu} when $m=n$ and $p_i,z_i\in \mathbb{R}$ for all $i=1,2,\cdots, n$.

\section{Applications in control theory}\label{sec:apps}
The sufficient condition for logarithimic complete monotonicity that we have derived in Corollary~\ref{cor:complex} also ensures that a given transfer function is externally positive. 
In this section, we exploit this fact to characterize {externally positive} systems and propose a controller design procedure which guarantees that the closed-loop system has a monotonic step response. 

\subsection{Characterization of {externally positive} systems}\label{subsec:Characterization_of_externally_positive_systems}
A linear time-invariant system is externally positive if and only if its transfer function is externally positive. 
Finding a complete characterization of the pole-zero patterns for which (\ref{eqn:transfer_function}) is externally positive is still an open problem~\cite{ScL2020}. However many sufficient conditions have been derived in the literature and for lower-order systems, most of the externally positive transfer functions are exposed using a combination of these results. For the sake of convenience, we provide conditions that characterize the full family of first- and second-order externally positive functions. Since these functions are a superset of the LCM functions, these conditions are necessary for a function to be LCM.

\begin{proposition}\label{prop:ex_pos_n=1}
Let $n=1$. The transfer function (\ref{eqn:transfer_function}) is externally positive if and only if $K>0$ and $B(p^{\downarrow}_1)\geq 0$.
\end{proposition}
\begin{proof}
{When $m=0$, the condition $B(p^{\downarrow}_1)\geq 0$ is already satisfied as $B(p^{\downarrow}_1)=K$ and when $m=1$, it can be written as $K(p^{\downarrow}_1-z_1)\geq 0$. In either case, the system is externally positive, according to \cite{lunzpa}.}
\end{proof}

\begin{proposition}\label{prop:ex_pos_n=2}
Let $n=2$ with no zero-pole cancellations. The transfer function (\ref{eqn:transfer_function}) is externally positive if and only if $p_1,p_2\in\mathbb{R}$, $K>0$, $B(p^{\downarrow}_1)\geq 0$, $B'(p^{\downarrow}_1)\geq 0$ and $B(p^{\downarrow}_1)\geq B(p^{\downarrow}_2)$ where $B'(.)$ is the numerator derivative.
\end{proposition} 
\begin{proof}
According to \cite{Dar2003}, $H(s)$ is not externally positive if it has complex-conjugate poles. Therefore, we assume $p_1,p_2\in\mathbb{R}$ and split the proof in two cases:

\noindent\textit{Case 1:} $p_1\neq p_2$. Following the partial fraction expansion and inverse Laplace transform, one has
    \begin{equation}\label{eqn:h_PFE_p1!=p2}
        h(t)=K\delta(t) \mathbf{1}_{[n=m]}+k_1\exp(p^{\downarrow}_1t)+k_2\exp(p^{\downarrow}_2t)
    \end{equation}
where $\mathbf{1}$ is the indicator function. Assuming $t>0$ in (\ref{eqn:h_PFE_p1!=p2}) and dividing both sides by $\exp(p^{\downarrow}_1t)$ indicates that $h(t)\geq 0$ is met if and only if
\begin{equation}\label{eqn(2):h_PFE_p1!=p2}
k_1+k_2\exp((p^{\downarrow}_2-p^{\downarrow}_1)t)\geq 0    
\end{equation}
The supremum and infimum of the left hand side of (\ref{eqn(2):h_PFE_p1!=p2}) are achieved at $t=0$ and $t\to +\infty$. Hence, (\ref{eqn(2):h_PFE_p1!=p2}) is satisfied for all $t>0$ if and only if it is satisfied for $t=0$ and $t\to +\infty$. This is equivalent to requiring that $k_1\geq0$ and $k_1+k_2\geq0$ hold true, where the residues $k_1$ and $k_2$ can be written as
\begin{equation}\label{eqn:residues}
    k_1=B(p^{\downarrow}_1)/(p^{\downarrow}_1-p^{\downarrow}_2),\quad k_2=B(p^{\downarrow}_2)/(p^{\downarrow}_2-p^{\downarrow}_1)
\end{equation}
This gives the conditions $B(p^{\downarrow}_1)\geq 0$ and $B(p^{\downarrow}_1)\geq B(p^{\downarrow}_2)$.

In case $n=m$, condition $K>0$ is also required for (\ref{eqn:transfer_function}) to be externally positive according to (\ref{eqn:h_PFE_p1!=p2}) and condition $B'(p^{\downarrow}_1)\geq 0$ is automatically satisfied when $B(p^{\downarrow}_1)\geq B(p^{\downarrow}_2)$. In case $n>m$, both $K>0$ and $B'(p^{\downarrow}_1)\geq 0$ are implied from $B(p^{\downarrow}_1)\geq 0$ and $B(p^{\downarrow}_1)\geq B(p^{\downarrow}_2)$.

\noindent\textit{Case 2:} $p_1= p_2$. The proof is similar to the previous case. Following the partial fraction expansion and inverse Laplace transform, one has
    \begin{equation}\label{eqn:h_PFE_p1=p2}
        h(t)=K\delta(t) \mathbf{1}_{[n=m]}+k_1\exp(p^{\downarrow}_1t)+k_2 t \exp(p^{\downarrow}_2t)
    \end{equation}
Assuming $t>0$ in (\ref{eqn:h_PFE_p1=p2}) and dividing both sides by $\exp(p^{\downarrow}_1t)=\exp(p^{\downarrow}_2t)$ indicates that $h(t)\geq 0$ is equivalent to $k_1+k_2t\geq 0$, which is satisfied for all $t>0$, if and only if both $k_1$ and $k_2$ are non-negative. This gives the conditions $k_1=B'(p^{\downarrow}_1)\geq 0$ and $k_2=B(p^{\downarrow}_1)\geq 0$. Condition $K>0$ is also required for (\ref{eqn:transfer_function}) to be externally positive according to (\ref{eqn:h_PFE_p1=p2}) when $n=m$. Otherwise in case $n>m$, $K>0$ is implied from $B(p^{\downarrow}_1),B'(p^{\downarrow}_1)\geq 0$.
\end{proof}

Characterization of higher-order {externally positive} systems ($n\geq 3$) is complicated and no general conditions exist that do not impose additional assumptions~\cite{JGC2001}. Corollary~\ref{cor:complex} provides sufficient conditions for a transfer function to be LCM. Since all LCM functions are externally positive, these conditions characterize a family of {externally positive} systems of arbitrary order with a simple set of conditions. But how conservative is such an approach? We try to answer this question in the following {Proposition.}

{\begin{proposition}\label{prop:LCM=EP}
All LCM transfer functions are externally positive. All externally positive transfer functions with $n\leq 2$ are LCM, but there are externally positive transfer functions with $n\geq 3$ that are not LCM.
\end{proposition}}
\begin{proof}To prove the equivalence of externally positive and LCM transfer functions for $n\leq 2$, we show that when $H(s)$ in (\ref{eqn:transfer_function}) satisfies the hypothesis of Proposition~\ref{prop:ex_pos_n=1} or Proposition~\ref{prop:ex_pos_n=2}, it also satisfies that of Lemma~\ref{lem:N&S_sumexp}. This is straight-forward for $n=1$ and for $m<n=2$. For $n=m=2$ the conditions in Proposition~\ref{prop:ex_pos_n=2} can be written as
\begin{equation}\label{eqn:expos_n=m=2}
\left\lbrace\begin{array}{l}
    p_1+p_2 \geq z_1+z_2 \\
    (p^{\downarrow}_1-z_1)(p^{\downarrow}_1-z_2)\geq 0
\end{array}\right.
\end{equation}
we split the rest of the proof in two cases:

\noindent\textit{Case 1:} $z_1,z_2\in\mathbb{R}$. 
In this case, condition (\ref{eqn:expos_n=m=2}) is equivalent to $p\succ_w z$. Thereby, using Proposition~\ref{prop:Karamata} with $g(x)=\exp(xt)$ results in (\ref{eqn:exp(pt)>exp(zt)}) and therefore, when function $H(s)$ is externally positive, it is also LCM.

\noindent\textit{Case 2:} $z_1=\bar{z}_2\in\mathbb{C}\backslash \mathbb{R}$. 
For the right hand side of (\ref{eqn:exp(pt)>exp(zt)}) one has
\begin{align}
    \exp(z_1t)+\exp(z_2t)&=2\exp(\operatorname{Re}(z_1)t)\cos(\operatorname{Im}(z_1)t)\nonumber\\
    &\leq 2\exp(\operatorname{Re}(z_1)t) \nonumber\\
    &\leq 2\exp((p_1+p_2)t/2) \label{AAA}\\
    &\leq \exp(p_1t)+\exp(p_2t) \label{BBB}
\end{align}
where in the steps (\ref{AAA}) and (\ref{BBB}) we have used (\ref{eqn:expos_n=m=2}) and the Jensen inequality. This proves (\ref{eqn:exp(pt)>exp(zt)}) and therefore, we conclude that when $n\leq 2$, the sets of externally positive functions and LCM functions coincide.

The second fact for higher-order systems with $n>2$ can be shown by noting that there exists transfer functions with complex-conjugate zeros on the right side of the dominant pole which are externally positive but {not LCM}. One such example is the third-order {transfer function
$$ H(s)=(s+0.5-i)(s+0.5+i)/(s+0.8)(s+1)(s+1.2) $$
which is externally positive according to \cite{LiF1997} but is not LCM, since condition (b) in Proposition~\ref{prop:necessary_conditions} asserts that LCM functions cannot have zeros with real parts greater than their dominant pole(s).}
\end{proof}

Given this one-sided relation between {LCM and externally positive} functions, we can only hope for the sufficient conditions provided in Theorem~\ref{thm:p^mu>z^mu} and Corollary~\ref{cor:complex} to expose a proper subset of {externally positive} systems. 

This subset can, however, be very sharp. {For example, Corollary~\ref{cor:complex} with $\mu=1$ exposes all the first and second-order externally positive systems with real zeros and poles. To see this,} it is enough to show that if a first or second-order function with real zeros and poles is LCM, it also satisfies the conditions in Corollary~\ref{cor:complex}, {according to Proposition~\ref{prop:LCM=EP}}. Here, we show this in the case $n=m=2$ and the other cases can be proved in a similar way.
Firstly, assuming $t\to +\infty$ yields $\sum_{i=1}^2\exp(p_it)\sim \exp(p^{\downarrow}_1 t)$ and $\sum_{i=1}^2\exp(z_it)\sim \exp(z^{\downarrow}_1 t)$, using which in (\ref{eqn:exp(pt)>exp(zt)}) implies
\begin{equation}\label{eqn:pmax>zmax}
    p^{\downarrow}_1\geq z^{\downarrow}_1
\end{equation}
In addition, when function (\ref{eqn:transfer_function}) is LCM, condition (c) in Proposition~\ref{prop:necessary_conditions} also holds true. This together with (\ref{eqn:pmax>zmax}) gives $w\succ_w v$ with $\mu=1$ and any $\delta$, where $w$ and $v$ are defined in (\ref{eqn:w,v}). Therefore, the conditions in Corollary~\ref{cor:complex} are also satisfied.

For higher-order systems ($n\geq 3$), {although Corollary~\ref{cor:complex} may not expose all the externally positive systems, one can} increase $\mu$ to get a tight characterization of the parameters that result in a externally positive system, as shown by the following example.

\begin{example}\label{ex:mu_effect}
Consider the transfer function $$H(s)=\frac{(s+10)(s+15)(s+30)}{(s+5)(s-p_2)(s-p_3)}$$
We are interested in the values of $p_2$ and $p_3$ such that $H(s)$ is externally positive.
{First consider the case in which $p_2$ and $p_3$ are real. Without loss of generality, we assume
$$
-35 \leq p^{\downarrow}_3 \leq p^{\downarrow}_2 \leq p^{\downarrow}_1=-5
$$
and use Theorem~\ref{thm:p^mu>z^mu} with $\delta=35$.}
Choosing $\mu=1$ in Theorem~\ref{thm:p^mu>z^mu} gives the same conditions as in \cite{Bal1994} which results in the yellow region shown in Figure~\ref{fig:mu_effect_a}. Choosing $\mu=2$, however, expands the region to coincide with the exact region associated with {non-negativeness of the impulse response}, revealing all the possible poles that make $H(s)$ externally positive in the real poles case. {This is evident in Figure~\ref{fig:mu_effect_a} as the region corresponding to $\mu=2$ (red) is equivalent to all the pairs $(p^{\downarrow}_2,p^{\downarrow}_3)$ satisfying condition (\ref{eqn:necexpos_sumpsumz}) which is known to be necessary for $H(s)$ to be externally positive.} Figure~\ref{fig:mu_effect_a} also features the sufficient conditions {for a system to be externally positive} offered by \cite{LiB2008} and \cite{lunzpa} in blue and gray colors respectively. Assuming the poles $p_2$ and $p_3$ to be complex-conjugates, we can use Corollary~\ref{cor:complex} to reveal the regions where system $H(s)$ is externally positive. Note that the conditions in \cite{LiB2008,lunzpa,Bal1994} are not applicable to complex-conjugate poles.
\end{example}

\begin{figure}
     \centering
     \begin{subfigure}[b]{0.49\textwidth}
        \centering
        \includegraphics[width=0.75\linewidth]{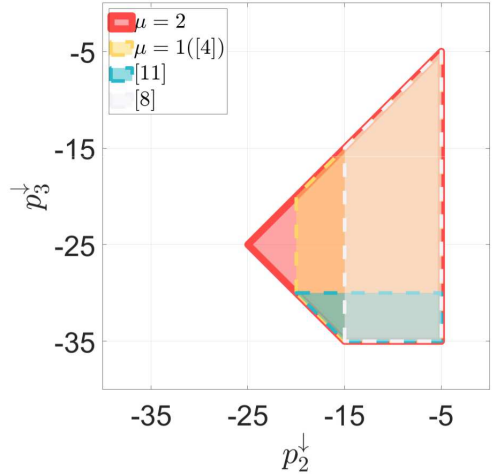}
	    \caption{Real poles.}
         \label{fig:mu_effect_a}
     \end{subfigure}
     \hfill
     \begin{subfigure}[b]{0.49\textwidth}
         \centering
        \includegraphics[width=0.75\linewidth]{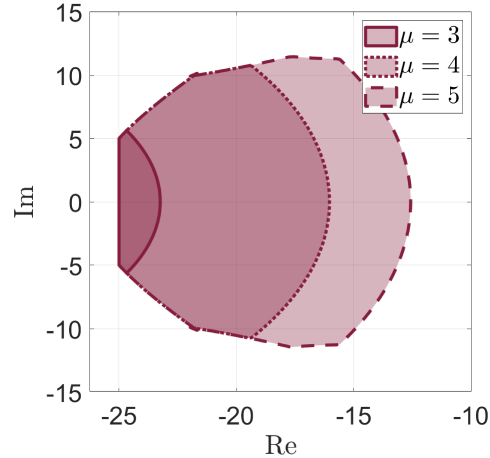}
	    \caption{Complex-conjugate poles.}
         \label{fig:mu_effect_b}
     \end{subfigure}
        \caption{{externally positive} regions in Example~\ref{ex:mu_effect}.}
\end{figure}


\subsection{Controller synthesis for monotonic tracking}\label{sec:control}
In Subsection~\ref{subsec:Characterization_of_externally_positive_systems}, it was shown that Corollary~\ref{cor:complex} (and its special case Theorem~\ref{thm:p^mu>z^mu}) is able to define a relatively sharp inner approximation of the set of {externally positive} systems represented by their zeros and poles. We will now show how these results can be used to develop an output-feedback synthesis procedure that 
guarantees monotonic tracking of reference changes with no steady-state error.

{In order to also optimize the closed-loop system performance, it is preferable to have a convex characterization of LCM functions, so that the optimal controller synthesis problem can be solved using reliable numerical routines.}
%
While the conditions in Theorem~\ref{thm:p^mu>z^mu} are not generally convex in $z,p$ (see Figure~\ref{fig:5th_order}), they become so by means of a variable change, {as shown in the following proposition.}

\begin{proposition}\label{prop:convex_real}
{Let $\delta$ be a fixed number. The set of vectors $\pi=(p+\delta)^{\mu\downarrow}\in(0,+\infty)^n$ and $z\in(-\delta,\infty)^n$ that satisfy the hypothesis of Theorem~\ref{thm:p^mu>z^mu} is a convex set.}
\end{proposition}
\begin{proof}
{See appendix~\ref{sec:proof_prop_1st}.}
\end{proof}

Next, we remark the convexity of the set of LCM functions described by Corollary~\ref{cor:complex}.

\begin{proposition}\label{rem:convex}
Let the shifted zeros and poles angles be fixed in the range
{
\begin{equation}\label{eqn:angle_shifted_poles}
    \vert \theta_i\vert, \vert \phi_j\vert <
    \left\lbrace \begin{array}{ll}
    \pi/2,      &\mu=1 \\
    \pi/2(\mu-1), & \mu>1
    \end{array}\right.
\end{equation}}
for all $i=1,\cdots,n$ and $j=1,\cdots,m$. Then for a fixed $\delta$, {the set of variables $\pi=w^{\downarrow}\in[0,+\infty)^{n+m}$, $v_i\geq 0$ ($1\leq i\leq n$) and $\zeta_i=v_{n+i}^{1/\mu}>0$ ($1\leq i\leq m$) that satisfy the conditions in Corollary~\ref{cor:complex} is a convex set.
}

\end{proposition}
\begin{proof}
{See appendix~\ref{sec:proof_prop_2nd}.}
\end{proof}

Note that the requirement (\ref{eqn:angle_shifted_poles}) can always be satisfied by increasing $\delta$. Proposition~\ref{rem:convex} makes it possible to use the obtained LCM conditions in {a controller synthesis procedure based on} convex optimization without any relaxation or additional conservatism. 

{Consider the set-up in Figure~\ref{fig:control_scheme}, where a} plant $H(s)=B(s)/A(s)$ in (\ref{eqn:transfer_function}) is controlled using the two-degree of freedom strategy
\begin{align*}
    G(s)U(s) &= K_c R(s) - F(s)Y(s).
\end{align*}
This controller structure was considered in \cite{TDJ2021} for discrete-time systems with the same objectives. We follow a similar approach here, choosing the polynomials
\begin{equation}\label{eqn:F,G}
F(s)=\sum_{k=0}^{n_c} f_k s^{n_c-k}, \quad G(s)=\sum_{k=0}^{n_c} g_k s^{n_c-k}    
\end{equation}
and the gain $K_c \in \mathbb{R}$ such that the closed-loop system 
\begin{align}\label{eqn:H^cl}
H^{cl}(s)&=\frac{B^{cl}(s)}{A^{cl}(s)}=\frac{K_c B(s)}{B(s)F(s)+A(s)G(s)}\nonumber\\
&=\frac{ b^{cl}_0 s^{n+n_c}+b^{cl}_1 s^{n+n_c-1}+ \cdots +b^{cl}_{n+n_c}}{ s^{n+n_c}+a^{cl}_1 s^{n+n_c-1}+ \cdots +a^{cl}_{n+n_c}}\nonumber\\
&=K_c K \frac{\prod_{i=1}^{m} (s-z_i)}{\prod_{i=1}^{n+n_c} (s-p^{cl}_i)}
\end{align}
is stable and has {a non-negative impulse response with} zero steady-state tracking error.

\begin{figure}
	\begin{center}
    \includegraphics[width=1\linewidth]{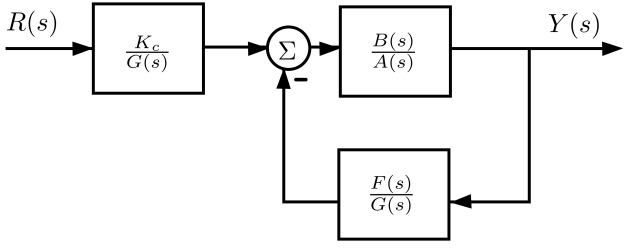}  
	\caption{Output feedback control scheme.}
	\label{fig:control_scheme}
	\end{center}
\end{figure}

As can be observed from (\ref{eqn:H^cl}), the controller leaves the closed-loop zeros at the locations of the (open-loop) plant. Therefore{,} by assuming
$$
n_c=n-1
$$
and that $B(s)$ and $A(s)$ are relatively prime, Sylvester's theorem (see \emph{e.g.}~\cite[Lemma~7.1]{GGS2001}) ensures that we can place the $2n-1$ closed-loop poles arbitrarily without affecting the closed-loop zeros. We first choose
$$\delta>-\min_{j}\lbrace\operatorname{Re}(z_j)\rbrace$$
to determine the region
$$p^{cl}_i\in\lbrace z\in \mathbb{C}\;\vert -\delta<\operatorname{Re}(z)<0\rbrace,
\quad i=1,2,\cdots, 2n-1$$
in which we would like to place the closed-loop poles. This region is chosen such that it is neither too restrictive (small $\delta$), nor {does it allow modes that are too fast} (large $\delta$) leaving the closed-loop system sensitive to noise. Next{,} we choose $\mu\in \mathbb{N}$ in a trade-off between conservatism of the LCM conditions (small $\mu$) and computational complexity (large $\mu$). After that, we choose the shifted closed-loop poles angles $\theta_i=\angle (p^{cl}_i+\delta)$ {as $\theta_i=0$ for $i=1,2,\dots,n_r$ and
$$
\vert \theta_i\vert <\left\lbrace
\begin{array}{ll}
    \pi/2,& \mu=1\\
    \pi/2(\mu-1),& \mu>1
\end{array}\right.
$$
for $i=n_r+1,\dots, 2n-1$, where $n_r$ is the desired number of real-valued closed-loop poles. For the non-real poles to be complex-conjugates, we make sure that $2n-1-n_r$ is an even number (\emph{i.e.}, $n_r$ is odd) and that $$\theta_{n_r+2j-1}=-\theta_{n_r+2j}$$ holds where $j=1,2,\dots (2n-n_r-1)/2$}. Then, we proceed with the variables {$w,v\in \mathbb{R}^{2n+m-1}$} in (\ref{eqn:w,v}) instead of {$p^{cl}$} in the synthesis procedure. This is because while the conditions in Corollary~\ref{cor:complex} are not convex {in the poles}, they are convex in the variables $v, w^{\downarrow}$. This requires enforcing the following affine constraints to the synthesis to ensure a correct change of variables:
\begin{equation}\label{eqn:affine_conds}
\left\lbrace
    \begin{array}{ll}
    w_i,v_i\geq 0, & i=1,2,\dots,2n-1+m \\
    w^{\downarrow}_i=0, & i=n_r+1, \dots,2n-1+m \\
    v_i=0, & i=1, \dots, n_r\\
    v_{2n-1+i}=\vert z_i+\delta\vert^\mu, & i=1,2,\dots,m \\
    v_{n_r+2i-1}=v_{n_r+2i}, & i=1,\dots (2n-n_r-1)/2
    \end{array}
\right.
\end{equation}

We are then ready to choose the {poles magnitudes
$$
\vert p^{cl}_i+\delta\vert = \left\lbrace
\begin{array}{rl}
w_i^{1/\mu}, & 1\leq i \leq n_r \\
v_i^{1/\mu}, & n_r+1\leq i \leq n
\end{array}
\right.
$$
to optimize an appropriate convex objective function $\psi(w,v)$, such that the closed-loop transfer function is stable and LCM according to Corollary~\ref{cor:complex}. The cost function $\psi(w,v)$ can be chosen to meet an additional performance objective. For example, one may choose
\begin{equation}\label{eqn:cost_robustness}
    \psi(w,v)=\vert w^{\downarrow}_1 - \max_{z_i\in\mathbb{R}}(z_i+\delta)^\mu\vert
\end{equation}
which matches the low frequency plant zero with the corresponding closed-loop pole, which helps to achieve a low complementary sensitivity function norm \cite{AsM2010}.}

{Finally, the whole synthesis process} can be formulated as the following standard convex optimization problem {with the decision variables $w^{\downarrow}$ and $v$:
\begin{align}\label{eqn:opt}
    \begin{array}[c]{rll}
    \underset{w^{\downarrow},v}{\text{minimize}} & \psi(w^{\downarrow},v)\\
    \mbox{subject to}
    & w^{\downarrow}_1 \leq \delta^\mu -\epsilon\\
    & w\succ_w v \\
    & (w,v,\theta,\phi) \textnormal{ satisfies } (\ref{eqn:psum(p)>psum(z),1_mu_proper})\\
    & (w,v) \textnormal{ satisfies } (\ref{eqn:affine_conds})
    \end{array}
\end{align}
Here, the first constraint ensures closed-loop stability while the second and third constraints ensure a closed-loop monotonic response. In (\ref{eqn:opt}), $\epsilon$ is a fixed positive number added to make the inequality non-strict, because most, if not all, convex optimization solvers only support non-strict inequalities. Finally, note that to verify the third constraint, $n$ is replaced by $2n-1$ in (\ref{eqn:psum(p)>psum(z),1_mu_proper}) adapting to the closed-loop system order.} The optimal solution $w^{\star},v^{\star}$ to (\ref{eqn:opt}) can be converted back to the zero-pole domain {using (\ref{eqn:w,v}):
\begin{equation}\label{eqn:w,v_backto_pcl}
p^{cl}_j=\left\lbrace
\begin{array}{ll}
    (w^{\star\downarrow}_j)^{1/\mu}-\delta, &  j=1,\cdots, n_{r}\\
    (v^{\star}_j)^{1/\mu}\exp(i\theta_j)-\delta, & j=n_r+1,\cdots, 2n-1
\end{array}\right.
\end{equation}
and the corresponding closed-loop characteristic equation coefficients $a^{cl}$ can be determined via the identity 
\begin{equation}\label{eqn:acl,pcl}
\sum_{k=0}^{2n-1}a^{cl}_k s^{2n-1-k}=\prod_{i=1}^{2n-1}(s-p^{cl}_i)
\end{equation}
We are now ready to compute the controller coefficients in (\ref{eqn:F,G}), by solving the following linear algebraic equation
\begin{equation}\label{eqn:M[f,g]=a^cl}
M\begin{bmatrix} f\\g\end{bmatrix}=a^{cl}
\end{equation}
where
$M\in\mathbb{R}^{2n\times 2n}$ has the elements
$$
[M]_{ij}=\left\lbrace
\begin{array}{lll}
b_{i-j}, & 1\leq j\leq n & j\leq i\leq j+n \\
a_{i-j+n}, & n+1\leq j\leq 2n & j-n\leq i\leq j \\
0, & \textnormal{otherwise} &
\end{array}\right.
$$
Equation (\ref{eqn:M[f,g]=a^cl}) has always a unique solution $f,g${~\cite[Lemma~7.1.]{GGS2001}}. Finally, the static gain $K_c$ is set to
\begin{equation}\label{eqn:Kc}
K_c=\left(B(0)F(0)+A(0)G(0)\right)/B(0)
\end{equation}
to give the closed-loop system from $R$ to $Y$ a stationary gain of one. 
The proposed synthesis procedure is summarized in Algorithm~\ref{alg}.

{
Although the proposed design procedure uses the same controller structure as \cite{TDJ2021} and places the closed-loop poles to ensure the closed-loop system is externally positive, it is different in all other respects. The approach in \cite{TDJ2021} considers discrete-time systems, relies on decomposing the closed-loop system into a series connection of first and second-order transfer functions, and proposes a manual procedure for placing the closed-loop poles on the real axis. The approach in this paper optimizes the closed-loop poles jointly based on the novel LCM conditions and allows for complex closed-loop poles.}

\begin{algorithm} 
\caption{Optimal pole-placement for monotonic controller synthesis:}\label{alg}
\begin{algorithmic}
\Require $K\in \mathbb{R}, z\in \mathbb{R}^m, p\in \mathbb{R}^n$ \Comment{Input data}
\Require $\mu \in \mathbb{N}$, $\delta$, $\theta$, $\psi(.)$ \Comment{Tuning parameters}
\State $w^{\star},v^{\star} \gets$(\ref{eqn:opt}) \Comment{solve (\ref{eqn:opt})}
\State $p^{cl} \gets w^{\star},v^{\star}$ \Comment{restore the poles via (\ref{eqn:w,v_backto_pcl})}
\State $f,g \gets p^{cl}$ \Comment{via (\ref{eqn:acl,pcl}) and (\ref{eqn:M[f,g]=a^cl})}
\State $K_c \gets f,g$ \Comment{via (\ref{eqn:Kc})}
\end{algorithmic}
\end{algorithm}

The next example demonstrates the power of the proposed synthesis method.

\begin{example}\label{ex:control}
{We are interested in stabilizing the second-order plant \cite{MDB2012}
\begin{equation}\label{eqn:H(s).control.ex}
    H(s)=\frac{s+2}{s^2+0.8s-0.2}
\end{equation}
using output-feedback. First, we consider cascade compensators $C(s)$ following the control law
\begin{equation}\label{cascade}
    U(s)=C(s)\left(R(s)-Y(s)\right)
\end{equation}
In order to also obtain a critically-damped closed-loop system, two proportional controllers can be designed as
\begin{equation}\label{eqn:C0}
    C'_0(s)=6.1665
\quad\textit{and}\quad
C_0(s)=0.2335
\end{equation}
Note that monotonic tracking is not taken into account by the controllers (\ref{eqn:C0}). In order to achieve a monotonic closed-loop step response instead, the first-order controllers
\begin{equation}\label{eqn:C1}
    C'_1(s)=\frac{s+76.6311}{s+10.4821}
\quad\textit{and}\quad
C_1(s)=\frac{s+10.4501}{s+59.581}
\end{equation}
were proposed in \cite{MDB2012}. Yet, the controllers (\ref{eqn:C1}) were derived based on a necessary condition for externally positive and therefore, cannot guarantee monotonic tracking beforehand. Nevertheless, instead of (\ref{cascade}), we can use the control design presented in Section~\ref{sec:control} with $\delta=5$, $\mu=1$ and $\theta_i=0$ for all $i=1,2,3$ to ensure a monotonic tracking. As the cost function, we choose (\ref{eqn:cost_robustness}). The resulting closed-loop step response is plotted in Figure~\ref{fig:step_response}. For comparison, the closed-loop step responses obtained by using the controllers in (\ref{eqn:C0}) and (\ref{eqn:C1}) are also shown in the same figure, which are either non-monotonic or more sluggish. In addition, the controller designed in this paper yields a smaller sensitivity peak ($M_s=1$) compared to the controllers $C_0(s)$, $C'_1(s)$ and $C_1(s)$ which respectively result in $Ms=1.1678$, $1.3403$ and $1.2127$. The sensitivity peak given by $C'_0(s)$ is also equal to $M_s=1$.}
\end{example}

\begin{figure}
\begin{center}
        \includegraphics[width=1\linewidth]{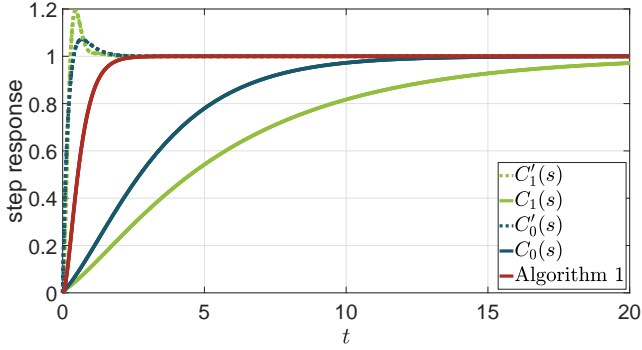}
	\caption{{The step responses of the closed-loop system in Example~\ref{ex:control} using different controllers.}}
	\label{fig:step_response}
\end{center}
\end{figure}

\section{Conclusions}\label{sec:con}

We have investigated the class of logarithmically completely monotonic (LCM) rational transfer functions. {Several conditions to determine when a transfer function is LCM were provided and expressed in terms of the transfer function’s zeros and poles. This includes conditions that are sufficient (Theorem~\ref{thm:p^mu>z^mu} and Corollary~\ref{cor:complex}), necessary (Proposition~\ref{prop:necessary_conditions}) and both necessary and sufficient (Lemma~\ref{lem:N&S_sumexp}).} It was shown that LCM rational functions are a subset of the space of externally positive transfer functions. As such, the LCM property was shown to be useful for providing strong and computationally tractable conditions to ensure that a transfer function has a non-negative impulse response. These results were then used to design output-feedback control loops that monotonically track reference changes without steady-state errors. Compared to existing approaches for non-overshooting reference tracking, the presented approach offered {a monotonic closed-loop response} guarantee in an output-feedback setting by blending pole-placement with convex optimization. {Future work will explore relaxing the conservatism of the conditions further and applying the synthesis procedure to practical problems.}

\appendix
\section{Appendix}
\subsection{Proof of Proposition~\ref{prop:convex_real}}\label{sec:proof_prop_1st}
{
\noindent 
Define the following variables in $\mathbb{R}^n$:
$$
\pi \defeq (p+\delta)^{\mu\downarrow},
\quad
\zeta \defeq (z+\delta)^{\mu}
$$
where the order of components is only fixed for $\pi$, \emph{i.e.},
\begin{equation}\label{eqn:x_ordering}
    \pi_1\geq \pi_2 \geq \cdots \geq \pi_n
\end{equation}
The majorization inequality $(p+\delta) ^{\mu} \succ_w  (z+\delta)^{\mu}$ can be equivalently expressed in terms of $\pi$ and $\zeta$ using the following intersection of $n$ conditions
\begin{equation}\label{eqn:mjrzn_rewrite}
\left\lbrace
\begin{array}{l}
\pi_1 \geq \max\lbrace \zeta_1,\zeta_2,\dots,\zeta_n \rbrace \\
\pi_2+\pi_2 \geq \max\lbrace \zeta_1+\zeta_2,\zeta_1+\zeta_3,\dots,\zeta_{n-1}+\zeta_n\rbrace \\
\vdots \\
\pi_1+\pi_2+\dots+\pi_n \geq \zeta_1+\zeta_2+\dots+\zeta_n
\end{array}
\right.
\end{equation}
Moreover,  $\delta>-\min_{i,j}\lbrace p_i,z_j\rbrace$, can be expressed as
\begin{equation}\label{eqn:z,p+mu>0}
    z_i+\delta >0 \textnormal{ and } p_i+\delta >0, \quad i=1,2,\dots,n
\end{equation}
In the range $x\in (0,+\infty)$, the monomial function $x^{\mu}$ is convex in $x$. Therefore, $\zeta_i=(z_i+\delta)^{\mu}$ is a convex function of $z$. As the point-wise maximum of convex functions is also convex, the right-hand side of each inequality in (\ref{eqn:mjrzn_rewrite}) is a convex function of $z$. In addition, as the left side of each inequality in (\ref{eqn:mjrzn_rewrite}) is linear in $\pi$, all the conditions in (\ref{eqn:mjrzn_rewrite}) are convex in the variables $\pi$ and $z$.
Next we show that (\ref{eqn:psum(p)>psum(z),1_mu}) is also a convex condition in the same variables. Inequality (\ref{eqn:psum(p)>psum(z),1_mu}) can be written as
\begin{equation}\label{eqn:x,z,1_mu}
    \sum_{i=1}^n \pi_i^{k/\mu} \geq \sum_{i=1}^n (z_i+\delta)^k  
\end{equation}
where $k\in \lbrace 1,2,\cdots,\mu -1\rbrace$. From (\ref{eqn:z,p+mu>0}) and the fact that the monomial function $x^{k/\mu}$ is concave in the range $x\in(0,+\infty)$, the left-hand side of (\ref{eqn:x,z,1_mu}) is a concave function of $\pi$. Using a similar arguement, from (\ref{eqn:z,p+mu>0}) and the fact that the monomial function $x^{k}$ is convex in the range $x\in (0,+\infty)$, we concude that the right-hand side of (\ref{eqn:x,z,1_mu}) is a convex function of $z$. This proves that inequality (\ref{eqn:x,z,1_mu}) is also a convex condition in $\pi$ and $z$.
As the intersection of several convex conditions (\ref{eqn:x,z,1_mu}), (\ref{eqn:mjrzn_rewrite}) and (\ref{eqn:x_ordering}), the condition given by Theorem~\ref{thm:p^mu>z^mu} is convex in the variables $\pi=(p+\delta)^{\mu\downarrow}$ and $z$.
\hfill$\square$}

\subsection{Proof of Proposition~\ref{rem:convex}}\label{sec:proof_prop_2nd}
{
\noindent
Define the variables $\pi\in\mathbb{R}^{n+m}$ and $\zeta\in\mathbb{R}^m$ where:
\begin{align*}
\pi &\defeq w^{\downarrow},\\
\zeta_i &\defeq v_{n+i}^{1/\mu}, \quad i=1,2,\dots,m
\end{align*}
where the order of components is only fixed for $\pi$, \emph{i.e.},
$$
    \pi_1\geq \pi_2 \geq \cdots \geq \pi_{n+m}
$$
We will show that the condition of Corollary~\ref{cor:complex} is convex in the variables $\pi$, $v_i$ ($i=1,2,\dots,n$) and $\zeta$. First, inequality (\ref{eqn:psum(p)>psum(z),1_mu_proper}) can be equivalently written as
\begin{equation}\label{eqn:pi,y,v,1_mu_proper}
    \sum_{i=1}^n \pi_i^{k/\mu}+\sum_{i=1}^n v_i^{k/\mu}\cos(\theta_i k) \geq \sum_{i=1}^{m} \zeta_i^{k}\cos(\phi_i k) 
\end{equation}
Since $\delta>-\min_{i,j}\lbrace \operatorname{Re}(p_i),\operatorname{Re}(z_j)\rbrace$, we have
\begin{align*}
    \pi_i\geq 0,\quad & i=1,2,\dots,n+m \\
    v_i\geq 0,\quad & i=1,2,\dots,n \\
    \zeta_i\geq 0,\quad & i=1,2,\dots,m
\end{align*}
Also, the monomial function $x^{k/\mu}$ ($k\in \lbrace 1,2,\cdots,\mu -1\rbrace$) is concave in the range $x\in[0,+\infty)$. Therefore, both the functionals $\pi_i^{k/\mu}$ and $v_i^{k/\mu}$ are concave, where $i=1,2,\dots,n$. 
On the other hand, Condition (\ref{eqn:angle_shifted_poles}) ensures
\begin{equation}\label{eqn:cos>0}
    \cos(\theta_i k),\cos(\phi_i k)\geq 0
\end{equation}
in (\ref{eqn:pi,y,v,1_mu_proper}). Therefore, the left side of (\ref{eqn:pi,y,v,1_mu_proper}) is a concave function of $\pi$ and $v_i$ ($1\leq i\leq n$). For the right side, we note that the monomial function $x^{k}$ ($k\in \lbrace 1,2,\cdots,\mu -1\rbrace$) is convex in the range $x\in[0,+\infty)$. Hence all the functionals $\zeta_i^k$ ($i=1,2,\dots,m$) are convex. This along with (\ref{eqn:cos>0}) proves that the right-hand side of inequality (\ref{eqn:pi,y,v,1_mu_proper}) is a convex function of $\zeta$. Thus, we conclude that inequality (\ref{eqn:pi,y,v,1_mu_proper}) is a convex condition in the variables $\pi$, $v_i$ ($1\leq i\leq n$) and $\zeta$. Next, we show that inequality $w \succ_w v$ is also a convex condition. This
inequality can be written as
\begin{equation}\label{eqn:expanded_w>v}
\sum_{i=1}^k w^{\downarrow}_i \geq \max_{\omega \in \Omega_k}\left\lbrace \sum\nolimits_{i=1}^{n+m} \omega_i v_i\right\rbrace
\end{equation}
for $k=1,2,\dots, n+m$, where
$$
\Omega_k=\left\lbrace \omega\in \lbrace 0,1\rbrace^{n+m}\vert
\sum\nolimits_{i=1}^{n+m} \omega_i=k
\right\rbrace
$$
which, in terms of $\pi$, $v_i$ ($1\leq i\leq n$) and $\zeta$ is given by
\begin{equation}\label{eqn:intermsofzeta}
\sum_{i=1}^k \pi_i \geq \max_{\omega \in \Omega_k}\left\lbrace \sum\nolimits_{i=1}^{n} \omega_i v_i
+\sum\nolimits_{i=1}^{m} \omega_{n+i} \zeta_i^{\mu}
\right\rbrace
\end{equation}
The right-hand side of (\ref{eqn:intermsofzeta}) is a convex function of $v_i$ ($1\leq i\leq n$) and $\zeta$ because it is the point-wise maximum of several convex functions, each represented by an $\omega\in \Omega_k$. Since the left-hand side of (\ref{eqn:expanded_w>v}) is linear in $\pi$, inequality (\ref{eqn:expanded_w>v}) is also a convex condition. Hence, the intersection of conditions (\ref{eqn:expanded_w>v}) and (\ref{eqn:pi,y,v,1_mu_proper}) is convex and therefore, the set of variables that satisfy the conditions in Corollary~\ref{cor:complex} is a convex set.
\hfill$\square$}

\bibliography{arxiv}
\bibliographystyle{ieeetr}

\end{document}